\documentclass[12pt]{article}
\usepackage{hyperref,amssymb,amsmath,graphicx,verbatim,eufrak,amsthm}
\usepackage{color,wrapfig}
\usepackage{calrsfs}
\usepackage{perpage} 
\MakePerPage{footnote} 

\hypersetup{pdfborder={0 0 0}, colorlinks=true, pdfborderstyle={},
  linkcolor=black, citecolor=black, urlcolor=blue}

\gdef\SetFigFontNFSS#1#2#3#4#5{} 

\newtheorem{thm}{Theorem}

\newtheorem{lem}{Lemma}
\newtheorem*{lem*}{Lemma}

\newtheorem*{thm*}{Theorem}

\theoremstyle{definition}

\newtheorem*{dfn*}{Definition}
\theoremstyle{remark}

\newcommand{\set}[1]{\left\{#1\right\}}

\newcommand{\Integer}{\mathbb{Z}}

\newcommand{\Z}{\Integer}

\newcommand{\R}{\mathbb{R}}
\newcommand{\C}{\mathbb{C}}

\newcommand{\eps}{\varepsilon}

\renewcommand{\Re}{\mathrm{Re}\,}
\renewcommand{\Im}{\mathrm{Im}\,}

\DeclareMathOperator{\E}{\mathbb{E}}     

\renewcommand{\Pr}{}
\let\Pr\relax
\DeclareMathOperator{\Pr}{\mathbb{P}}

\newcommand{\1}[1]{\mathbf{1}_{\set{ #1 } }}
\newcommand{\ov}[1]{\overline{#1}}

\newcommand{\ignore}[1]{ }

\newcommand{\fancyf}{\mathcal}

\newcommand{\Aa}{\fancyf{A}}
\newcommand{\Bb}{\fancyf{B}}
\newcommand{\Cc}{\fancyf{C}}
\newcommand{\Dd}{\fancyf{D}}
\newcommand{\Ee}{\fancyf{E}}

\newcommand{\note}[1]{ }

\newcommand{\SLE}{\mathrm{SLE}}

\begin{document}

\title{I knew I should have taken that left turn at Albuquerque}

\author{Gady Kozma\thanks{Faculty of Mathematics and Computer Science, The
Weizmann Institute of Science, Rehovot 76100, Israel. Email:
\texttt{gady.kozma@weizmann.ac.il}}
\and
Ariel Yadin\thanks{Department of Mathematics,
Ben Gurion University,
Be'er Sheva 84105, Israel.
Email: \texttt{yadina@bgu.ac.il} }
}

\date{ }

\maketitle

\begin{abstract}
We study the Laplacian-$\infty$ path as an extreme case of the Laplacian-$\alpha$
random walk.  Although, in the finite $\alpha$ case, there is reason to believe
that the process converges to $\SLE_\kappa$, with $\kappa = 6/(2 \alpha+1)$, we show
that this is not the case when $\alpha=\infty$.  In fact, the scaling limit depends heavily
on the lattice structure, and is not conformal (or even rotational) invariant.
\end{abstract}

\section{Introduction}

In recent years, much study has been devoted to the phenomena
of conformally invariant scaling limits of processes in $\Z^2$,
the two-dimensional Euclidean lattice.
The invention of Schramm-Loewner Evolution ($\SLE$) \cite{Schramm00SLE}, and
subsequent development, have lead to many new results regarding such limits.

The first process considered by Schramm was loop-erased random
walk, or LERW. This is a process in which one considers a random walk
(on some graph) and then erases the loops in the path of that walk,
obtaining a self-avoiding path.  LERW was first defined by Lawler
\cite{Lawler80LERWdfn}.  In \cite{LSW04LERW}, Lawler, Schramm and
Werner proved that the scaling limit of LERW on $\Z^2$ is $\SLE_2$.

LERW is related to another process,
the so called {\em Laplacian-$\alpha$ random walk}, defined in
\cite{LEP86}.  In fact, LERW and Laplacian-$1$ random walk are the
same process \cite{L87}.  For completeness, let us define the
Laplacian-$\alpha$ random walk.

Let $\alpha \in \R$ be some real parameter.
Let $G = (V,E)$ be a graph, and let $w$ be a vertex (the target).
Let $S \subset V$ be a set not containing $w$.
Let $f_{w,S;G}:V \to [0,1]$ be the function defined
by setting $f_{w,S;G}(x)$ to be the probability
that a random walk on $G$ started at $x$ hits $w$
before $S$. $f_{w,S;G}$ is $1$ at $w$, $0$ on $S$ and harmonic in $G
\setminus (S \cup \set{w})$, and if the graph $G$ is finite and
connected, then it is the unique function satisfying these three conditions.
Hence $f_{w,S;G}$ is usually called the {\em solution to the Dirichlet
problem in $G$ with boundary conditions $1$ on $w$ and $0$ on $S$}.

\begin{dfn*}[Laplacian-$\alpha$ random walk]
Let $G$ be a graph.
Let $s \neq w$ be vertices of $G$.  The Laplacian-$\alpha$ random walk
on $G$, starting at $s$ with target $w$, is the process $(\gamma_t)_{t \geq 0}$
such that $\gamma_0 = s$, and such that for any $t>0$ the distribution of
$\gamma_t$ given $\gamma_0,\gamma_1,\ldots,\gamma_{t-1}$ is
$$ 
\Pr [ \gamma_t = x \ | \ \gamma_0,\gamma_1,\ldots,\gamma_{t-1} ] =
\1{ x \sim \gamma_{t-1} } \cdot
\frac{ f^\alpha(x) }{ \sum_{y \sim \gamma_{t-1}} f^{\alpha}(y) } , $$
where $f = f_{w,\gamma[0,t-1];G}$ is the solution to
the Dirichlet problem in $G$ with boundary conditions
$1$ on $w$ and $0$ on $\gamma[0,t-1] = \set{ \gamma_0, \gamma_1,
\ldots, \gamma_{t-1} }$ and where $x\sim y$ means that $x$ and $y$ are
neighbors in the graph $G$.
The process terminates when first hitting $w$. Here and below we use
the convention that $0^\alpha = 0$ even for $\alpha \leq 0$.
\end{dfn*}

As already remarked, the case $\alpha=1$ is equivalent to LERW in any
graph, and therefore in two dimensional lattices has $\SLE_2$ as its
scaling limit. Another case which is understood is the case $\alpha=0$
which is simply a random walk which chooses, at each step, equally among
the possibilities which do not cause it to be trapped by its own
past. Examine this process on 
the hexagonal (or honeycomb) lattice. This is a lattice with degree 3 so
the walker has at most 2 possiblities at each step.
A reader with some patience will be able to resolve some
topological difficulties and convince herself that this process is
exactly equivalent to an exploration of critical percolation on the
{\em faces} of the hexagonal lattice (say with black-white boundary
conditions and the edges between boundary vertices unavailable to the
Laplacian random walk\footnote{To the best of our knowledge this was
  first noted in \cite{LawlerLaplacian_b}, in the last paragraph of \S
 2.}. By \cite{S01, W09}, this has $\SLE_6$ as its
scaling limit.

In \cite{LawlerLaplacian_b}
Lawler gives an argument that leads one to expect that the scaling limit of the Laplacian-$\alpha$
random walk on $\Z^2$ should be $\SLE_{\kappa}$, for $\kappa= \tfrac{6}{2 \alpha+1}$,
for the range of parameters $\alpha > -1/2$.
In a public talk about this heuristic argument given in Oberwolfach in
2005 (which GK attended) Lawler stated (paraphrasing) that the
argument can be trusted less and less as $\alpha$ increases. Therefore
it seems natural to stress it as far as possible by setting
$\alpha=\infty$. 
Let us define the process formally.

\begin{dfn*}[Laplacian-$\infty$ path]
Let $G$ be a graph.
Let $s \neq w$ be vertices of $G$.  The Laplacian-$\infty$ path
on $G$, starting at $s$ with target $w$, is the path $(\gamma_t)_{t \geq 0}$
such that $\gamma_0 = s$, and such that for any $t>0$,
given $\gamma_0,\gamma_1,\ldots,\gamma_{t-1}$,
we set $\gamma_t$ to be the vertex $x \sim \gamma_{t-1}$
that maximizes $f_{w,\gamma[0,t-1];G}(x)$ over
all vertices adjacent to $\gamma_{t-1}$.
If there is more than one maximum adjacent to $\gamma_{t-1}$,
one is chosen uniformly among all maxima.
The path terminates when first hitting $w$.
\end{dfn*}

Note that except for the rule in the case of multiple maxima,
the Laplacian-$\infty$ path is not random.

The conjecture that the Laplacian-$\alpha$ random walk converges to $\SLE_\kappa$ for $\kappa = \tfrac{6}{2 \alpha+1}$
naturally leads one to ask whether this also holds for $\alpha = \infty$; that is,
does the Laplacian-$\infty$ path on $\delta \Z^2$ converge to the
(non-random) path $\SLE_0$, as $\delta$ tends to $0$?
Specifically, if this is true, the conformal invariance of $\SLE_0$ hints
that this should hold regardless of the lattice one starts with,
or at least for any rotated version of $\Z^2$.

We will show that, perhaps surprisingly, this is not the case. In
fact, the process on $\Z^2$ can be described almost
completely. Without further ado let us do so

\begin{thm}
\label{thm:main thm}
There exists a universal constant $C$ such that for any
    $(a,b)\in\Z^2$ with $a>|b|\ge C$ the following holds. The
    Laplacian-$\infty$ path starting at $(0,0)$ with target $(a,b)$ has
    $\gamma_t=(t,0)$ for all $t$ with probability 1.

If $a=b \geq C$ then $\gamma_t=(t,0)$ with probability $1/2$ 
and $\gamma_t = (0,t)$ with probability $1/2$.
\end{thm}

In other words (and using the symmetries of the problem), the walker
does the first step in the correct direction but then continues forward,
missing the target (unless the target is extremely close to the axis which is the path of the walker)
and goes on to infinity, never ``turning left''. If the target is on a diagonal the walker
chooses among the two possible first steps equally. Comparing to
$\SLE_0$, which is a deterministic (conformal image of a) straight
line from the start to the target, we see that the process is indeed
deterministic, and is indeed a straight line, but is not (necessarily)
aimed at the target, is not rotationally invariant and is not
independent of the lattice --- rotated versions of $\Z^2$ give rise to
different scaling limits.



To explain the reason for this behavior in a single sentence, one may
say that the pressure of the past of the process outweighs the pull
of the target. For those interested in the proof, let us give a rough
description of the ideas involved by applying them to prove the
following lemma.
\begin{lem*} Let $t>2$, and let $x\ge 1$. Let $p_y$ be the probability
  that a random walk starting from some $y\in\Z^2$ avoids the interval
  $[(-x,0),(0,0)]$ up to time $t$. Then
\[
p_{(1,0)} > (1+c)p_{(0,1)}
\]
for some absolute constant $c>0$.
\end{lem*}
\begin{proof}[Proof sketch] Couple two walkers starting from these two
  points so that their paths are a reflection through the diagonal
  $\{(x,x):x\in\Z\}$ until they first hit, and then they move
  together. This shows that $p_{(1,0)}\ge p_{(0,1)}$. Further, since
  it is possible for the walker starting from $(0,1)$, in 2 steps, to
  hit $(-1,0)$ without the other walker hitting the forbidden
  interval, we see that the difference $p_{(1,0)}-p_{(0,1)}$ is of the
  same order of magnitude as each of them.
\end{proof}

\subsection{Generalizations and speculations}
Although we use a planar argument,
some simple adaptations of our methods
should work also for higher dimensions; i.e.
for Laplacian-$\infty$ paths on $\Z^d$ (instead of reflecting through
a diagonal, one needs to reflect through a hyperplane orthgonal to a
vector of the form $e_1\pm e_i$ where $e_i$ is the
$i$\textsuperscript{th} standard basis vector).
Also, slight variations on the methods used can
produce a similar result for the Laplacian-$\infty$ path
on the triangular lattice.

There is some awkwardness in our comparison to SLE since we prove our
results on the whole plane, where SLE is not well defined. To rectify
this one might examine our process in a large domain $\Dd$, directed at one
point $w$ on its boundary (i.e.~solve the Dirichlet problem with boundary
conditions 0 on the path $\gamma$ and on $\partial\Dd\setminus w$ and 1 on
$w$). This process may be readily compared to (time reversed) radial $\SLE_0$. While we
cannot analyze this process until the time it hits $w$, our methods do
show that the process is a straight line along one of the axes until
almost hitting the boundary of $\Dd$, which is very far from radial
$\SLE_0$ which should be the \emph{conformal image} of a straight line
from $0$ to $w$, which is a smooth path but not necessarily a straight
line, and definitely not necessarily aligned with one of the
axes. Another natural variation is starting from the boundary namely letting
$s\in\partial \Dd$ and solving the Dirichlet problem with 0 on $\gamma
\cup \partial \Dd$ and 1 on some $w\in\Dd$. Analyzing this process
using our methods requires some more assumptions on $\Dd$ but it does
work, for example, for $\Dd$ being a square and $s$ not too close to
one of the corners. One gets that the path is a straight line
perpendicular to the boundary almost until hitting the facing
boundary, at which point our analysis no longer works, but again, this
is quite enough to see that the process is very far from $\SLE_0$
(radial or chordal --- depending on whether $w$ is inside $\Dd$ or on
its boundary). We
will not prove either claim as they are similar to those that we do prove
with only some minor additional technical difficulties.

We did some simulations on the behavior of the process on a $600\times 600$
torus. Here the process must hit the target (this is easily seen on any
finite graph).
\begin{figure}
\centering
\input{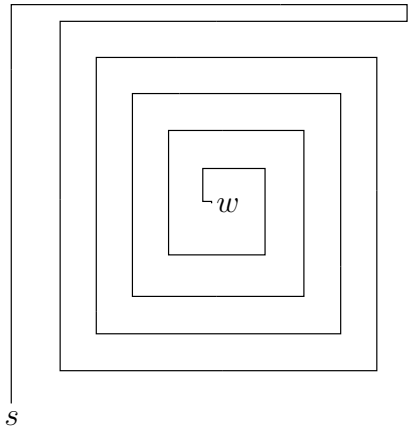}
\caption{Laplacian-$\infty$ on a $600\times 600$ torus.}
\label{fig:600}
\end{figure}
See Figure \ref{fig:600}. As one can see the process does hit the
target but takes its time to do so, turning only when it is about to
hit its past. Some aspects of the picture could definitely do with
some explanation: why does the process turn around quickly after the
first round (the very top of the picture)? We have no proof and only a
mildly convincing heuristic explanation for this behavior.

For an explanation of the name of this paper, see \cite{B45}.

\subsection{Acknowledgements}
We wish to thank N. Aran for pointing us to \cite{B45}. The
torus simulations would not have been possible without Timothy Davis'
SuiteSparseQR, a library for fast solution of sparse self-adjoint
linear equations.

\section{Proof}

\subsection*{Notation.}
$\Z^2$ denotes the discrete two dimensional Euclidean lattice;
we denote the elements of $\Z^2$ by their complex counterparts,
e.g. the vector $(1,2)$ is denoted by $1+2i$.

$\Pr_x$ and $\E_x$ respectively, denote the measure and expectation of
a simple (nearest-neighbor, discrete time) random walk
on $\Z^2$, $(X_t)_{t \geq 0}$, started at $X_0 = x$.
For a set $S \subset \Z^2$, we denote by $T(S)$ the hitting time of $S$; that is
$$ T(S) = \inf \set{ t \geq 0 \ : \ X_t\in S } , $$
Occasionally we will use $T(S)$ for a subset $S\subset \C$ that is
not discrete, and in this case the hitting time of $S$ is the first
time the walk passes an edge that intersects $S$.
We also use the notation
$T(z,r) = T(\set{ w \ : \ |w-z| \ge r})$, the exit time from the ball
of radius $r$ centered at $z$ (we will always use it with the starting point inside
the ball).
For a vertex $z \in \Z^2$ we use the notation $T(z) = T(\set{z})$.

We will denote universal positive constants with $c$ and $C$ where $c$ will
refer to constants ``sufficiently small'' and $C$ to constants
``sufficiently large''. We will number some of these constants
for clarity.


We begin with an auxiliary lemma.

\begin{lem}
\label{lem:I,D vs. I}
There exists a universal constant $C>0$ such that the following holds.
Let $w \in \Z^2$.
Let $D = \set{ x+ix \ : \ x \in \Z}$ be the discrete diagonal.
Let $I = [-x,0] \cap \Z$, for some $x >0$.
Then,
\begin{equation}\label{eq:lemIDI}
\Pr_i [ T(w) < T(I\cup D) ] \leq C |w|^{-1/2} \Pr_i [ T(w) < T(I) ] .
\end{equation}
\end{lem}

Before starting the proof we need to apologize for some of the choices
we made. It is well known that the probability that a
random walk escape from a corner of opening angle $a$ to distance $r$
is of the order $r^{-\pi/a}$. The case of $a=2\pi$ was famously done by
Kesten\footnote{This result is not stated in \cite{K87}
  explicitly but the upper bound can be inferred from results proved
  there (particularly lemma 6) easily, and the lower bound can be
  proved by the same methods.} \cite{K87} and a simpler proof can be found in the book
\cite[\S 2.4]{L91}. The general case can be done using
multiscale coupling to Brownian motion, but we could not find a
suitable reference, and including a full proof would have weighed
down on this paper. 
The reader is encouraged to verify that given the
general $r^{-\pi/a}$ claim, both sides of (\ref{eq:lemIDI}) can be calculated
explicitly.
Thus, the exponent $1/2$ on the right side of \eqref{eq:lemIDI} is not optimal,
but is sufficient for our purpose and the proof is far simpler.

\begin{proof}[Proof of Lemma \ref{lem:I,D vs. I}]
Let us recall the aforementioned results regarding escape
probabilities. 
Equations (2.37) and (2.38) of \cite[\S 2.4]{L91}
tell us that for any $r>0$,
\begin{equation}\label{eq:Kesten}
 \Pr_i [ T(0,r) < T(I) , \Re( X_{T(0,r)} ) \geq 0 ] \geq c_1 r^{-1/2}
 ,
\end{equation}
for some universal constant $c_1 >0$.
We will also need the probability of escape from the diagonal $D$. This
particular case is simple because for simple random walk the
two projections $\Re X + \Im X$ and $\Re X - \Im X$ are
\emph{independent} one dimensional random walks. This makes it easy to
calculate escape probabilities in a rhombus. Namely, if $S_r=\{x:|\Re
x| + |\Im x|=r\}$ then the question whether, for random walk starting
from $i$, $T(D)\le T(S)$ or not, is equivalent to the question whether
a one-dimensional random walk hits $1$ before hitting $r$ and before a
second, independent one-dimensional random walk hits $\pm r$. Both are
well known to be $\ge 1-C/r$ so all-in all we get
\begin{equation}\label{eq:diag}
 \Pr_i [ T(0,r) < T(D) ] \leq \Pr_i [ T(S_r)<T(D) ] \leq C_1 r^{-1}.
\end{equation}

Let $A(r,R) = \set{ z \in \C \ : \  r \leq |z| \leq R}$ denote the closed annulus
of inner radius $r$ and outer radius $R$.
Fix $r = \tfrac{|w|}{2}$.
Without loss of generality, by adjusting the constant in the statement
of the lemma,
we can assume that $r$ is large enough.
Let $A = A(r/4,r)$.  So $|w|>r$ and $w \not\in A$.
Let $V$ be the set of all
$v$ with $\Re(v) \ge 0$ such that $\Pr_i [ X_{T(0,r/2)} = v ] > 0 $.
So $r/2 \leq |v| \leq r/2+1$ and $v \in A$. Let $U$ be the set of all
$u \in \Z^2$ such that $\Pr_i [ X_{T(0,r/4-2)} = u ] >
0$. Specifically, $|u| < r/4$ and $u \not\in A$. See Figure
\ref{fig:uvw}, left.
\begin{figure}
\centering
\input{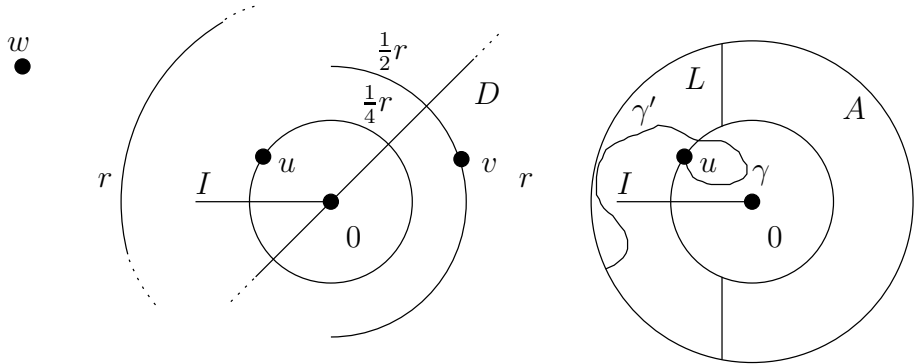}
\caption{On the left, $u$, $v$ and $w$. On the right, $L$, $\gamma$ and $\gamma'$.}
\label{fig:uvw}
\end{figure}

Fix $u \in U$ and $v \in V$.
Consider the function
$f(z) = \Pr_z [ T(w) < T(I) ]$.
This function is discrete-harmonic in the split ball $A(0,r) \setminus I$ and $0$ on $I$.
Thus, there exists a path $\gamma = (u=\gamma_0,\gamma_1,\ldots,\gamma_n)$ in $\Z^2$ from
$u$ to some $\gamma_n\not\in A(0,r)$ such that $f(\cdot)$ is non-decreasing on $\gamma$; i.e. $f(\gamma_{j+1}) \geq f(\gamma_j)$
for all $0 \leq j \leq n-1$. See Figure \ref{fig:uvw}, right.

We now examine the slightly-less-than-half of $A$, $L:=\{x\in A:\Re
(x) <-r/16\}$ and divide into two cases according to whether $\gamma\cap A$ is
contained in $L$ or not. In the second case, let $v'\in
\gamma\cap(A\setminus L)$. By the
discrete Harnack inequality \cite[Theorem 6.3.9]{LL10} we have
$f(v)\ge cf(v')$ for some absolute constant. To aid the reader in
using \cite{LL10} efficiently, here are the sets we had in mind:
\begin{equation*}
\left(\begin{aligned}
  &\text{$\mathbf{K}$ \& $\mathbf{U}$}\\
  &\text{from \cite{LL10}}
\end{aligned}\right)
\qquad
\begin{aligned}
\mathbf{K} &=\{x\in \R^2:\tfrac{1}{4}\le|x|\le 1
  \textrm{ and }x_1\ge -\tfrac{1}{16}\}\\
\mathbf{U} &=\{x\in \R^2:\tfrac{3}{16}<|x|<\tfrac{3}{2}
  \textrm{ and }x_1> -\tfrac{1}{8}\}
\end{aligned}
\end{equation*}
Note that $f(\cdot)$ is discrete harmonic in $\ov{r \mathbf{U}}$.
Hence, since $f$ is non-decreasing on $\gamma$, $f(v') \geq f(u)$, and in this case
\begin{equation}\label{eq:fvfu}
f(v)\ge cf(u).
\end{equation}

Showing (\ref{eq:fvfu}) in the case that $\gamma\cap A\subset L$ is
only slightly more complicated. Let $\gamma'$ be the last portion of
$\gamma$ in $L$ i.e.\ $\{\gamma_{m+1},\dotsc,\gamma_n\}$ where $m<n$ is
maximal such that $\gamma_m\not\in A$ (see Figure \ref{fig:uvw},
right). In this case $\gamma'$ divides $L$ into two components, and
$I\cap L$ lies completely in one of them. Assume for concreteness
it is in the bottom one. Then every path crossing $L$ counterclockwise
will hit $\gamma$ before hitting $I$. Examine therefore the event
$\Ee$ that random walk starting from $v$ will exit the slit
annulus $A\setminus\{x:\Re (x) = -r/16, \Im (x) <0 \}$ by hitting the slit
from its left side. By the invariance principle \cite[\S 3.1]{LL10}, if $r$ is
sufficiently large then $\Pr(\Ee)>c_2$ for some constant $c_2>0$
independent of $r$, uniformly in $v\in V$. However, $\Ee$ implies that the random
walk traversed $L$ counterclockwise, hence it hits $\gamma$ before
hitting $I$. We get
$$ \Pr_v [ T(\gamma) < T(I\cup\{w\}) ] > c_2. $$
Since $f(X_t)$ is a martingale up to the first time $X_{\cdot}$ hits $I$ or
$w$, we may use the strong Markov property at the stopping time $T(\gamma)$
to get
\[
f(v)\ge c_2\E [ f(X_{T(\gamma)})\,|\, T(\gamma)<T(I\cup \{w\}) \ge c_2
  f(u)
\]
i.e.~we have established (\ref{eq:fvfu}) in both cases.

We now use the bounds on escape probabilities above to get
\begin{align}
&&\lefteqn{\Pr_i [ T(w) < T(I\cup D) ]}\qquad& \nonumber\\
&&& \leq
\Pr_i [ T(0,r/4-2) < T(I \cup D) ] \cdot \max_{u \in U} \Pr_u [ T(w) <
  T(I\cup D) ] \nonumber\\
&&& \leq \Pr_i [ T(0,r/4-2) < T(I \cup D) ] \cdot \max_{u \in U} f(u) \nonumber\\
&\textrm{\footnotesize By (\ref{eq:fvfu})}&& \leq
 \Pr_i [ T(0,r/4-2) < T(D) ] \cdot C_2 \min_{v \in V} f(v) \nonumber\\
&\textrm{\footnotesize By (\ref{eq:diag})}&& \leq
 C_3 r^{-1} \cdot \min_{v \in V} f(v)
\label{eq:234}
\end{align}
where $C_2,C_3>0$ are universal constants.
The lemma now follows from applying
the strong Markov property at the stopping time $T(0,r/2)$,
\begin{align*}
\Pr_i [ T(w) < T(I) ] &\geq \Pr_i [ T(0,r/2) < T(I) ,
  \Re(X_{T(0,r/2)}) \geq 0] \cdot \min_{v\in V}f(v)\\
\textrm{\footnotesize By (\ref{eq:Kesten})}\qquad& \geq
 c r^{-1/2} \min_{v \in V} f(v)\\
\textrm{\footnotesize By (\ref{eq:234})}\qquad&\geq cr^{1/2}\Pr_i[T(w)<T(I\cup D)].
\qedhere
\end{align*}
\end{proof}

We now turn to the main lemma, which uses the coupling argument
sketched in the introduction.

\begin{lem} \label{lem:Pr[T(w)<T(I)]}
There exist universal constants $C,\eps>0$ such that the following holds.
Let $I = [-x,0] \subset \R$, for some $x \geq 1$,
and let $w \in \Z^2$ such that $|w| >C$.
Then,
$$ \Pr_1 [ T(w) < T(I) ] > \Pr_i [ T(w) < T(I) ] (1+\eps) . $$
\end{lem}

(The proof will give $\eps=4^{-7}$.)

\begin{proof}
We couple two random walks on $\Z^2$ started at $1$ and $i$, by constraining them to be the mirror
image of each other around $D = \set{ x+ix \ : \ x \in \Z }$ until
they meet. When they do, they glue and continue walking together. In
formulas, given the random walk $(X_t)$,
let $(Y_t)$ be a random walk coupled to $(X_t)$ as follows.
Set $X_0 = 1$ and $Y_0 = i$.  For $t>0$, if $Y_{t-1} \neq X_{t-1}$,
let $Y_t = i \ov{X_t}$.
If $Y_{t-1} = X_{t-1}$, then let $Y_t = X_t$.
It is immediate that $(Y_t)$ is also a random walk.

Let $ \tau = \min \set{ t \geq 0 \ : \ X_t = Y_t } , $
be the coupling time.
For all $t \leq \tau$,
$\Re(Y_t) = \Im(X_t)$ and $\Im(Y_t) = \Re(X_t)$.
Hence, for any $t \leq \tau$, we have that $Y_t = X_t$ if
and only if $Y_t , X_t \in D$.
So we conclude that $\tau = T(D)$.

Now, let $T^1(I) = \min \set{ t \geq 0 \ : \ X_t \in I }$
and $T^i(I) = \min \set{ t \geq 0 \ : \ Y_t \in I }$ be the hitting times of $I$
for $X_{\cdot}$ and $Y_{\cdot}$ respectively.  Similarly,
let $T^1(w),T^i(w)$ be the hitting times
of $w$ for $X_{\cdot}$ and $Y_{\cdot}$ respectively.
Since $X_0=1$, we have that $D$ separates $X_0$ from $I$, so $\tau = T(D) \leq T^1(I)$.
Thus, $T^1(I) \geq T^i(I)$ always.

\begin{figure}
\centering
\input{ABC.pstex_t}
\caption{The events $\Aa$, $\Bb$ and $\Cc$.}
\label{fig:ABC}
\end{figure}

Let $\Aa$, $\Bb$ and $\Cc$ be the three events depicted in 
Figure \ref{fig:ABC}. Formally, let $\Aa$ be the event $\set{T(D)\leq T^1(w),T(D)\leq T^i(w)}$.
Let $\Bb$ be the event $\{T^1(w)<T(D)\leq T^i(w)\}$,
and let $\Cc$ be the event $\set{ T^i(w) < T(D) \leq T^1(w) }$. 
Note that $\Aa,\Bb$ and $\Cc$ are pairwise disjoint and their union is
the whole space. Furthermore, either $\Pr[\Bb] = 0$ or $\Pr[\Cc]=0$,
depending on whether $\Re(w) < \Im(w)$ or $\Re(w) > \Im(w)$
respectively (if they are equal both events are empty).

Now, on the event $\Aa$ we have that $T^i(w) = T^1(w)$.
Thus, on the event $\Aa$, the event $\set{T^i(w) < T^i(I)}$ is contained in
the event $\set{T^1(w) < T^1(I)}$.
Hence,
\begin{align}
\Pr [ T^1(w) < T^1(I) , \Aa ] & - \Pr [ T^i(w) < T^i(I) , \Aa ]
\nonumber \\
& = \Pr [ T^i(I) < T(D) \leq T^i(w) = T^1(w) < T^1(I) ]. \nonumber
\end{align}

\begin{figure}
\begin{centering}
\input{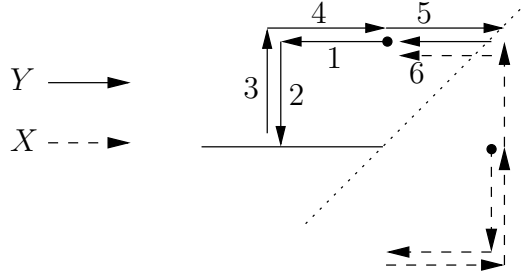}
\caption{The event that gives $4^{-6}$ in the proof.}
\label{fig:coupling}
\end{centering}
\end{figure}%

Next, consider the event
$\set{ Y_1 = i-1, Y_2 = -1, Y_3 = i-1 , Y_4 = i, Y_5 = 1+i , Y_6 = i}$
(which is the same as the event
$\set{ X_1 = 1-i, X_2 = -i, X_3 = 1-i , X_4 = 1, X_5 = 1+i , X_6 =
  i}$, see Figure \ref{fig:coupling}),
which implies that $T^i(I) < T(D) \leq T^i(w) = T^1(w)$.
(Here we use that $x \geq 1$, so $-1 \in I$.)
We have that
\begin{align}
\label{eqn:A}
\Pr [ T^1(w) < T^1(I) , \Aa ] & - \Pr [ T^i(w) < T^i(I) , \Aa ]
\geq  4^{-6} \Pr_i [ T(w) < T(I) ]  .
\end{align}


As for the event $\Bb$,
we have that $\Bb \subset \set{ T^1(w) < T^1(I) }$.
So,
\begin{equation}
\label{eqn:B}
\Pr [ T^1(w) < T^1(I) , \Bb ] -\Pr [ T^i(w) < T^i(I),\Bb] = \Pr [ \Bb
] - \Pr [ T^i(w) < T^i(I),\Bb] \ge 0
\end{equation}

Finally, the event $\Cc$ implies $T^i(w) < T(D)$ and therefore
\begin{multline}
\label{eqn:C}
\Pr [ T^1(w) < T^1(I) , \Cc ] - \Pr [ T^i(w) < T^i(I) , \Cc ] \ge\\
- \Pr [ T^i(w) < T^i(I) , \Cc ] \ge
 - \Pr_i [ T(w) < T(I\cup D) ].
\end{multline}


Combining \eqref{eqn:A}, \eqref{eqn:B} and \eqref{eqn:C}, we get that
\begin{align} \label{eqn:A,B,C}
\Pr_1 [ T(w) < T(I) ] & - \Pr_i [ T(w) < T(I) ] \nonumber \\
& \geq 4^{-6} \Pr_i [ T(w) < T(I) ] 
- \Pr_i [ T(w) < T(I\cup D) ] .
\end{align}

We have not placed any restrictions on the constant $C$ from the
statement of the lemma so far.
Let $C_4$ be the constant from Lemma \ref{lem:I,D vs. I}.
We now choose $C \geq \left(C_4 4^7\right)^2$.
Thus, if $|w| > C$ then by Lemma \ref{lem:I,D vs. I},
$$ \Pr_i [ T(w) < T(I\cup D) ] < 4^{-7} \Pr_i [ T(w) < T(I) ] . $$
Plugging this into \eqref{eqn:A,B,C} completes the proof of the lemma.
\end{proof}





The last piece of the puzzle is to determine the first step of
the Laplacian-$\infty$ path.

\begin{lem} \label{thm:gamma1}
Let $w \in \Z^2$ with $\Re (w)> |\Im w|$. Let $(\gamma_t)_{t \geq 0}$ be
the Laplacian-$\infty$ path on $\Z^2$, started at $\gamma_0 = 0$ with
target $w$. Then $\gamma_1=1$.

If $\Re (w) = \Im (w) > 0$ then $\gamma_1 = 1$ with probability $\tfrac{1}{2}$
to be and $\gamma_1 = i$ with probability $\tfrac{1}{2}$.
\end{lem}

\begin{proof}
We start with the case $\Re (w)>|\Im (w)|$. Recall that $\gamma_1$ is the neighbor $e$ of 0 that maximizes
the probability $\Pr_{e} [ T(w) < T(0) ]$.

%

As in the proof of Lemma \ref{lem:Pr[T(w)<T(I)]}, we couple two random walks
$(X_t),(Y_t)$, starting at $X_0 = 1$ and $Y_0=i$
respectively, by reflecting them around $D$.
We use $T^1(0),T^1(w),T^i(0),T^i(w)$ to denote the hitting times
of $0$ and $w$ by these walks, in the obvious way.
Recall from the proof of Lemma \ref{lem:Pr[T(w)<T(I)]},
that the coupling time of these walks is $T(D)$, the hitting time of $D$.

Since $D$ separates $w$ from $i$, we have that $T^1(w) \leq T^i(w)$.
Thus, the event $\{ T^i(w) < T^i(0) \}$ implies the event $\set{
  T^1(w) < T^1(0) }$. Further, this inclusion is strict --- the event
that $X$ hits $w$ before $D$ has positive probability. Hence
\[
\Pr_1 [ T(w) < T(0) ] > \Pr_i [ T(w) < T(0) ].
\]
Showing that $\Pr_1[ T(w) < T(0) ] > \Pr_e[ T(w) < T(0) ]$ for
$e=-1,-i$ is done likewise by reflecting through the imaginary line or
the opposite diagonal $D^*=\{x-ix:x\in\Z\}$, respectively. This
completes the proof of the case $\Re (w)>|\Im (w)|$.
For the case $\Re (w) = \Im (w) >0$, just note that the problem is now
symmetric to reflection through the diagonal $D$, so
$\Pr_1[T(w)<T(0)]=\Pr_i[T(w)<T(0)]$, so the walker chooses among them
equally. Both are larger than the probabilities at $-1$ and $-i$,
again by reflecting through the diagonal $D^*$.
\end{proof}


\begin{proof}[Proof of Theorem \ref{thm:main thm}]
We take $C>0$ so that Lemma \ref{lem:Pr[T(w)<T(I)]} holds with this
constant $C$.
We prove the theorem by induction on $t$.
Let $w = a+ib$.
The case of $t=1$ is handled
by Lemma \ref{thm:gamma1}. Assume therefore that $\gamma_s = s$
for all $0 \leq s \leq t-1$.
Let $I = \set{ \gamma_s \ : \ 0 \leq s \leq t-1 }$.
Let $f:\Z^2 \to [0,1]$ be the function $f(z) = \Pr_z [ T(w) < T(I) ]$.
Translating by $-(t-1)$,
since $|w-(t-1)| \geq |\Im(w)| > C$,
we can use Lemma \ref{lem:Pr[T(w)<T(I)]}, to get that
$f(t) > f(t-1+i)$. Reflecting through the real line and using Lemma
\ref{lem:Pr[T(w)<T(I)]} again we get $f(t) > f(t-1-i)$. Thus, $\gamma_t =
t$ and the theorem is proved.
\end{proof}

%
%
%
%
%
%
%
%
%
%



\end{document}